\documentclass[12pt,a4paper,notitlepage,reqno]{amsart}

\usepackage[british]{babel}   

\usepackage[margin=2.7cm]{geometry}

\usepackage{amsmath,amssymb,amsfonts,amsthm}
\usepackage{enumerate}
\usepackage{hyperref}
\usepackage{mathrsfs}
\usepackage{tikz}
\usepackage{tikz-cd}
\usepackage[centertableaux,smalltableaux]{ytableau}

 \usetikzlibrary{matrix}
 \usetikzlibrary{shapes}
 \usetikzlibrary{arrows, automata}
 \usetikzlibrary{decorations,decorations.pathmorphing,decorations.pathreplacing,decorations.markings}
 \usetikzlibrary{through}
 \usetikzlibrary{decorations.pathreplacing,calligraphy}

 \tikzset{every picture/.style={baseline=-.65ex} }
 \tikzset{ext/.style={circle, draw,inner sep=1pt},int/.style={circle,draw,fill,inner sep=1pt},nil/.style={inner sep=1pt}}

\tikzset{->-/.style={decoration={
    markings,
    mark=at position .5 with {\arrow{>}}},postaction={decorate}}}

\usepackage{bm}   

\theoremstyle{plain}
\newtheorem{theorem}{Theorem}
   \numberwithin{theorem}{section}

\newtheorem{proposition}[theorem]{Proposition}

\newtheorem{claim}[theorem]{Claim}
\newtheorem{example}[theorem]{Example}

\newtheorem*{theorem*}{Theorem}
\theoremstyle{definition}
\newtheorem{definition}[theorem]{Definition}

\theoremstyle{remark}
\newtheorem{remark}[theorem]{Remark}
\newtheoremstyle{italic-env}
{}                
{}                
{\slshape}        
{}                
{\bfseries}       
{}               
{0pt}               
{}                
\theoremstyle{italic-env}

\newtheoremstyle{plain-env}
{}                
{}                
{\upshape}        
{}               
{\bfseries}       
{}              
{0pt}               
{}                
\theoremstyle{plain-env}

\def \C {\mathbb C}
\def \Q {\mathbb Q}

\def \cM {\mathcal{M}}  
\def \cS {\mathcal{S}}

\def \d {\partial}

\def \Im {{\rm Im}\,}

\newcommand{\oleg}[1]{
  \begin{tikzpicture}
    \node (v) at (0,0) {$\omega$};
    \node (w) at (1,0) {#1};
    \draw (v) edge (w);
  \end{tikzpicture}
}
\newcommand{\eleg}[1]{
  \begin{tikzpicture}
    \node (v) at (0,0) {$\epsilon$};
    \node (w) at (1,0) {#1};
    \draw (v) edge (w);
  \end{tikzpicture}
}
\newcommand{\tripleo}{
  \begin{tikzpicture}
    \node[int] (i) at (0,.5) {};
    \node (v1) at (-.5,-.2) {$\omega$};
    \node (v2) at (0,-.2) {$\omega$};
    \node (v3) at (.5,-.2) {$\omega$};
  \draw (i) edge (v1) edge (v2) edge (v3);
  \end{tikzpicture}
}
\newcommand{\threelegs}[3]{
  \begin{tikzpicture}
    \node[int] (i) at (0,.5) {};
    \node (v1) at (-.5,-.2) {#1};
    \node (v2) at (0,-.2) {#2};
    \node (v3) at (.5,-.2) {#3};
  \draw (i) edge (v1) edge (v2) edge (v3);
  \end{tikzpicture}
}
\newcommand{\fourlegs}[4]{
  \begin{tikzpicture}
    \node[int] (i) at (0.25,.5) {};
    \node (v1) at (-.5,-.2) {#1};
    \node (v2) at (0,-.2) {#2};
    \node (v3) at (0.5,-.2) {#3};
    \node (v4) at (1,-.2) {#4};
  \draw (i) edge (v1) edge (v2) edge (v3) edge (v4);
  \end{tikzpicture}
}
\newcommand{\twotwolegs}[4]{
  \begin{tikzpicture}
    \node[int] (i) at (0,.5) {};
    \node[int] (j) at (1,.5) {};
    \node (v1) at (-.5,-.2) {#1};
    \node (v2) at (0,-.2) {#2};
    \node (v3) at (1,-.2) {#3};
    \node (v4) at (1.5,-.2) {#4};
  \draw (i) edge (v1) edge (v2) edge (j) (j) edge (v3) edge (v4);
  \end{tikzpicture}
}

\newcommand{\twoonetwolegs}[5]{
  \begin{tikzpicture}
    \node[int] (i) at (0,.5) {};
    \node[int] (j) at (0.75, 0.5) {};
    \node[int] (k) at (1.5,.5) {};
    \node (v1) at (-.5,-.2) {#1};
    \node (v2) at (0,-.2) {#2};
    \node (v3) at (0.75,-.2) {#3};
    \node (v4) at (1.5,-.2) {#4};
    \node (v5) at (2,-.2) {#5};
  \draw (i) edge (v1) edge (v2) edge (j) (j) edge (v3) edge (k) (k) edge (v4) edge (v5);
  \end{tikzpicture}
}

\newcommand{\fivelegs}[5]{
  \begin{tikzpicture}
    \node[int] (i) at (0.25,.5) {};
    \node (v1) at (-.75,-.2) {#1};
    \node (v2) at (-0.25,-.2) {#2};
    \node (v3) at (0.25,-.2) {#3};
    \node (v4) at (0.75,-.2) {#4};
    \node (v5) at (1.25,-.2) {#5};
  \draw (i) edge (v1) edge (v2) edge (v3) edge (v4) edge (v5);
  \end{tikzpicture}
}

\newcommand{\twothreelegs}[5]{
  \begin{tikzpicture}
    \node[int] (i) at (0,.5) {};
    \node[int] (j) at (1,.5) {};
    \node (v1) at (-.5,-.2) {#1};
    \node (v2) at (0,-.2) {#2};
    \node (v3) at (0.75,-.2) {#3};
    \node (v4) at (1.25,-.2) {#4};
    \node (v5) at (1.75,-.2) {#5};
  \draw (i) edge (v1) edge (v2) edge (j) (j) edge (v3) edge (v4) edge (v5);
  \end{tikzpicture}
}

\newlength{\fsize}
\def\fsize{12pt}    

\def\graphscale {0.7} 

\begin{document}

\title[Computation of the Excess Four Case]{Computation of Weight 11 Compactly Supported Cohomology of Moduli Spaces of Curves}

\author[L\'eon Burkhardt]{L\'eon Burkhardt}
\address{Departement of Mathematics, ETH Zurich, R\"amistrasse 101, 8092 Zurich, Switzerland}
\email{lburkhardt@student.ethz.ch}

\begin{abstract}
  In their homonymous article, Sam Payne and Thomas Willwacher construct a combinatorial graph complex to compute the weight 11 part of the compactly supported cohomology of the moduli space of curves $\cM_{g,n}$ and compute explicitly the cohomology of the introduced graph complexes in cases of complexes of excess zero, one, two, and three.
  In this paper, we extend the computation the cohomology to excess four graph complexes. 
  Along the way, we give more details on generators, the computation process, and the definition of the graph complex.
  \end{abstract}

  \maketitle

\setcounter{tocdepth}{2}  
\tableofcontents
\section{Introduction}

In their homonymous article, Sam Payne and Thomas Willwacher construct a combinatorial graph complex to compute the weight 11 part of the compactly supported cohomology of the moduli space of curves $\cM_{g,n}$.
The central result of the paper establishes a correspondence between these groups and the tensor product of cohomology groups of a graph complex $B_{g,n}$ constructed by the authors and cohomology of the Deligne-Mumford compactification $\overline{\cM}_{1,11}$. Let us note that $H^11 (\overline{\cM}_{1,11})$ is two dimensional \cite{CLP23}.
\begin{proposition}[\cite{PW23}, Proposition 1.7]\label{prop:PW-iso}
  The weight 11 compactly supported cohomology of $\cM_{g,n}$ is isomorphic to the tensor product of $ H^{11} (\overline{\cM}_{1,11})$ with the cohomology of $B_{g,n}$:
  $$gr_{11} H_c^k(\cM_{g,n}) \cong H^k(B_{g,n}) \otimes  H^{11} (\overline{\cM}_{1,11})$$
\end{proposition}

In particular, this new tool unveils the exponential growth of $H_c^{2g+k}(\cM_{g,0})$ for numerous new values of $k$. 
Beside these results, the authors also compute explicitly the cohomology of the introduced graph complexes in cases that are simple with regard to a new statistic named excess (see Section \ref{subsec:Excess} for more details). They show that graphs of low excess are relatively easy to determine, and compute the cohomology groups of complexes of excess zero, one, two, and three.

In this paper we build up on these computations and determine the cohomology of all graph complexes with excess four, yielding explicit results for various additional cohomology groups of moduli spaces of curves.

Our computations lead to the following explicit results for cohomology groups of excess four graph complexes:

\begin{theorem}\label{thm:cohomologyBgn}
  In the setting above, the cohomology of the graph complexes of excess four is given in terms of irreducible representations $V_\lambda$ of $S_n$ by:
  \vskip \fsize
  \begin{flushleft}
    $H^k(B_{1,13}) = 
  \begin{cases}
     V_{51^8} \oplus V_{3^21^7} \oplus V_{321^8} \oplus V_{2^21^9} & \text{ for } k= 13, \\
     0 & \text{ otherwise.}
  \end{cases}
  $
  \end{flushleft}
  
  \vskip \fsize
  \begin{flushleft}
    $H^k(B_{3,10}) = 
  \begin{cases}
     V_{31^7} & \text{ for } k=15,  \\
     V_{51^5} \oplus V_{51^5} \oplus V_{421^4} \oplus V_{41^6} \oplus V_{3^21^4} \oplus\\
      V_{3^21^4} \oplus V_{321^5} \oplus V_{321^5} \oplus V_{2^21^6} \oplus V_{2^21^6} & \text{ for } k=16, \\
     0  & \text{ otherwise.}
  \end{cases}$
  \end{flushleft}

  \vskip \fsize
  \begin{flushleft}
    $H^k(B_{5,7}) = 
  \begin{cases}
     V_{31^4} \oplus V_{31^4} & \text{ for } k=18,  \\
     V_{51^2} \oplus V_{51^2} \oplus V_{51^2} \oplus V_{421} \oplus V_{41^3} \oplus \\
     V_{3^21} \oplus V_{3^21} \oplus V_{321^2} \oplus V_{321^2} \oplus V_{2^21^3} \oplus V_{2^21^3} & \text{ for } k=19,  \\
     0 & \text{ otherwise.}
  \end{cases}$
  \end{flushleft}
  
  \vskip \fsize
  \begin{flushleft}
    $H^k(B_{7,4}) = 
  \begin{cases}
    V_{31} \oplus V_{31} & \text{ for } k= 21, \\
    V_{2^2} \oplus V_{2^2} & \text{ for } k=22,  \\
    0 & \text{ otherwise.}
  \end{cases}$
  \end{flushleft}
  
  \vskip \fsize
  \begin{flushleft}
    $H^k(B_{9,1}) = 
  \begin{cases}
    V_1 & \text{ for } k=22,  \\
    0 & \text{ otherwise.}
  \end{cases}$
  \end{flushleft}
  \vskip \fsize
\end{theorem}

\begin{remark}
  Note that these results coincide with the Euler characteristics computed in \cite{PW23}.
\end{remark}

As an immediate corollary, we obtain the weight 11 compactly supported cohomology of $\cM_{g,n}$ for 
$$(g,n) \in \{(1,13),(3,10),(5,7),(7,4),(9,1)\}$$
using the isomorphism in Proposition \ref{prop:PW-iso}.

The rest of this paper is structured as follows.
We begin in Section \ref{sec:Background} by a detailed presentation of the graph complex defined in by Payne-Willwacher and the tools used in our computations.
In Section \ref{sec:Generators}, we determine the module structure of our graph complex by determining all generators. These are given by combinations of the obtained connected components of suitable excess.
In Section \ref{sec:Cohomology}, we compute the boundary operator on our chain complex and compute the cohomology. We do so by using a spectral sequence arising naturally from the lexicographic order on partitions of 4 associated with the generators.

\vskip\fsize
\emph{Acknowledgements: I would like to thank my supervisor Professor Thomas Willwacher for proposing this subject to me and for his support throughout this project.}

\section{Background} \label{sec:Background}

In this section, we introduce the objects of our study. While we refer to \cite{PW23} for the original ideas, we aim here to give more details and insights about the constructions and computations. Our aim is to make the ideas more accessible to experts from other fields and young researchers beginning in this field.

\subsection{The Graph Complex}\label{subsec:GraphCx}

In this section, we study the graph complexes $B_{g,n}$ defined in \cite{PW23}, to which we refer the reader for full details.

\subsubsection{Generating Graphs}
$B_{g,n}$ is generated by connected graphs $\Gamma$ having one special vertex, $n$ labelled legs, and loop order $g-1$. The special vertex has valency of at least $2$. All other vertices have valency of at least $3$. We do not allow tadpoles except at the special vertex.
A subset of the half-edges incident to the special vertex is marked. We will refer to these half-edges as \emph{distinguished} in the remaining of this text and denote the cardinality of the subset of distinguished half-edges by $w$. Further, all edges of $\Gamma$ but the $n$ labelled legs will be called \emph{structural edges}. 

\begin{example}
  A typical generator $\Gamma$ of $B_{g,n}$ looks as depicted for instance in Figure \ref{fig:typicalGen} below. Here, the marked vertex is represented by two concentric circles and the distinguished half-edges are marked by arrows.
  
  \begin{figure}[ht]
    \begin{center}
      \resizebox{0.25\hsize}{!}{
        $\begin{tikzpicture}[scale=1]
          \node[ext,accepting] (v0) at (0,0){};
          \node[int] (v1) at (140:.75) {};
          \node[font = \small] (v6) at  (140:1.5) {8};
          \node[int] (v2) at (-90:0.7) {};       
          \node[int] (v7) at (-45:.7) {};
          \node[font=\small] (v3) at (10:1) {1};
          \node[font=\small] (v4) at (45:1) {$\ddots $};
          \node[font=\small] (v5) at (80:1) {7};
          \node[int] (v8) at (-160:1) {};
          \draw (v0) 
          edge[->-] (v5) edge[->-] (v4) edge[->-] (v3) 
          edge[->-,bend right] (v1) edge [-,bend left] (v1) 
          edge[->-,bend right] (v8) edge[->-] (v8) edge [->-,bend left] (v8) 
          edge[bend right] (v2) edge [->-] (v2) edge[->-, bend left] (v7) edge[->-] (v7);
          \draw (v7) edge (v2);
          \draw (v1) edge (v6);
          \end{tikzpicture}$
      }
    \end{center}
    \caption{A typical generator of $B_{g,n}$} \label{fig:typicalGen}
  \end{figure}
\end{example}

For readability reasons, we will use a representation in ``blown-up'' components to represent such graphs. This representation is obtained by removing the special vertex and listing the resulting connected components. We label  the distinguished half-edges with $\omega$ and the non-distinguished half-edges with $\epsilon$. Figure \ref{fig:blownup} illustrates this procedure.

\vskip\fsize
\begin{figure}[ht]
  \begin{center}
    \resizebox{0.9\hsize}{!}{
      $\begin{tikzpicture}[scale=1]
        \node[ext,accepting] (v0) at (0,0){};
        \node[int] (v1) at (140:.75) {};
        \node[font = \small] (v6) at  (140:1.5) {8};
        \node[int] (v2) at (-90:0.7) {};       
        \node[int] (v7) at (-45:.7) {};
        \node[font=\small] (v3) at (10:1) {1};
        \node[font=\small] (v4) at (45:1) {$\ddots $};
        \node[font=\small] (v5) at (80:1) {7};
        \node[int] (v8) at (-160:1) {};
        \draw (v0) 
        edge[->-] (v5) edge[->-] (v4) edge[->-] (v3) 
        edge[->-,bend right] (v1) edge [-,bend left] (v1) 
        edge[->-,bend right] (v8) edge[->-] (v8) edge [->-,bend left] (v8) 
        edge[bend right] (v2) edge [->-] (v2) edge[->-, bend left] (v7) edge[->-] (v7);
        \draw (v7) edge (v2);
        \draw (v1) edge (v6);
        \end{tikzpicture}
  \quad \quad \leftrightarrow \quad \quad 
  \oleg{1} \ldots \oleg{7} \threelegs{8}{$\omega$}{$\epsilon$} \tripleo \twotwolegs{$\epsilon$}{$\omega$}{$\omega$}{$\omega$}$
    }
  \end{center}
  \caption{Representation in blown up components} \label{fig:blownup}
\end{figure}

The graphs come with an orientation given by an ordering of the structural edges and of distinguished half-edges. The complete data of a generator is therefore of the form 
$$ (\Gamma, e_1 \wedge \cdots \wedge e_k \wedge h_1 \wedge \cdots \wedge h_w)$$
for $\Gamma$ a graph as described above and a choice of ordering $o = e_1 \wedge \cdots \wedge e_k \wedge h_1 \wedge \cdots \wedge h_w$ of the set of structural edges $\{e_1, \ldots, e_r\}$ and of distinguished half-edges $\{h_1, \ldots, h_w\}$.

Finally, we make the following identifications. Different orientations are identified with a sign given by the signature of the permutation that relates them. Additionally, we identify isomorphic graphs and quotient out the subcomplex of graphs with $w<11$.

\begin{remark}\label{rmk:symmetry}
  As an easy consequence of the definition of $B_{g,n}$ (in particular the identifications made) we note that graphs with component \scalebox{\graphscale}{$\oleg{$\omega$} $},  \scalebox{\graphscale}{$\threelegs{$\epsilon$}{$\omega$}{$\epsilon$} $} or  \scalebox{\graphscale}{$\twoonetwolegs{$\omega$}{$\omega$}{$\omega$}{$\omega$}{$\omega$}$} vanish. Indeed, if we keep track of a choice of labelling on the connected component, we obtain
  \vskip\fsize
\begin{center}  
  \resizebox{0.5\hsize}{!}{
    $
    \begin{tikzpicture}
      \node[int] (i) at (0,.5) {};
      \node (v1) at (-.5,-.2) {$\epsilon$};
      \node (v2) at (0,-.2) {$\omega$};
      \node (v3) at (.5,-.2) {$\epsilon$};
      \node[font=\tiny,text=blue] (l1) at (0.1,0) {1};
      \node[font=\tiny, text=blue] (l2) at (0.6,0) {2};
      \node[font=\tiny, text=blue] (l3) at (-.4,0) {3};
      \node[font=\tiny,text=green] (l4) at (0.15,-0.35) {1};
    \draw (i) edge (v1) edge (v2) edge (v3);
    \end{tikzpicture}
    =
    \begin{tikzpicture}
      \node[int] (i) at (0,.5) {};
      \node (v1) at (-.5,-.2) {$\epsilon$};
      \node (v2) at (0,-.2) {$\omega$};
      \node (v3) at (.5,-.2) {$\epsilon$};
      \node[font=\tiny,text=blue] (l1) at (0.1,0) {1};
      \node[font=\tiny, text=blue] (l2) at (0.6,0) {3};
      \node[font=\tiny, text=blue] (l3) at (-.4,0) {2};
      \node[font=\tiny,text=green] (l4) at (0.15,-0.35) {1};
    \draw (i) edge (v1) edge (v2) edge (v3);
    \end{tikzpicture}
    =
    -
    \begin{tikzpicture}
      \node[int] (i) at (0,.5) {};
      \node (v1) at (-.5,-.2) {$\epsilon$};
      \node (v2) at (0,-.2) {$\omega$};
      \node (v3) at (.5,-.2) {$\epsilon$};
      \node[font=\tiny,text=blue] (l1) at (0.1,0) {1};
      \node[font=\tiny, text=blue] (l2) at (0.6,0) {2};
      \node[font=\tiny, text=blue] (l3) at (-.4,0) {3};
      \node[font=\tiny,text=green] (l4) at (0.15,-0.35) {1};
    \draw (i) edge (v1) edge (v2) edge (v3);
    \end{tikzpicture}
    ,$
  }
\end{center}
  using in the first equality the identification of isomorphic generators and in the second the identification of permuted ordering with sign. The same argument holds for the two other mentioned examples. More generally, graphs with such an ``odd'' symmetry vanish in $B_{g,n}$. Note however, that we do not allow the isomorphisms to permute the labelling of the legs, whence components such as \scalebox{\graphscale}{\threelegs{1}{$\omega$}{2}} are non-trivial in $B_{g,n}$.
\end{remark}

\subsubsection{The Chain Complex}
We define the degree of a graph generator to be given by the difference between the number of structural edges and the number of distinguished half-edges with some shift. We write this as
$${\deg (\Gamma) =  22 + |E| -n - w } ,$$
where $E$ is the set of all edges of $\Gamma$.

\vskip\fsize
The boundary operator of the graph complex is given by a combination of two differentials. We describe these as presented in \cite{PW23}.

The first part, denoted $\delta_\omega$, simply removes one distinguished half-edge from the orientation.
\[
    \delta_\omega (\Gamma, e_1\wedge\cdots \wedge e_k \wedge h_1\wedge \cdots \wedge h_w)
    =\sum_{j=1}^r (-1)^{k+j-1}
    (\Gamma, e_1\wedge \cdots \wedge e_k \wedge h_1 \wedge \cdots \hat h_j \cdots \wedge h_w)
\]
Pictorially:
\begin{align*}
    \delta_\omega:
    \begin{tikzpicture}
        \node[ext,accepting] (v) at (0,0){};
        \draw (v) edge[->-] +(0:.6) edge[->-] +(60:.6) edge +(120:.6) edge +(180:.6) edge +(240:.6) edge[->-] +(300:.6);
    \end{tikzpicture}
    &\mapsto
    \sum\pm
    \begin{tikzpicture}
        \node[ext,accepting] (v) at (0,0){};
        \draw (v) edge[->-] +(0:.6) edge +(60:.6) edge +(120:.6) edge +(180:.6) edge +(240:.6) edge[->-] +(300:.6);
    \end{tikzpicture} \quad .
\end{align*}
The piece $\delta_s$ acts by splitting vertices.
For convenience we shall further decompose $\delta_s=\delta_s^\bullet+\delta_s^\circ$ into a piece $\delta_s^\circ$ splitting the special vertex and $\delta_s^\bullet$ splitting the other vertices. The part $\delta_s^\bullet$ is defined by
\[
\delta_s^\bullet (\Gamma, e_1\wedge\cdots \wedge e_k \wedge h_1\wedge \cdots \wedge h_w)
=
\sum_{v\in V_\bullet\Gamma} \sum_{\text{split v}} (\Gamma', e_0\wedge e_1\wedge\cdots \wedge e_k \wedge h_1\wedge \cdots\wedge h_w)    \, ,
\]
where the outer sum is over all non-special vertices $v$ of $\Gamma$. The inner sum is over all admissible ways of replacing the vertex $v$ by two vertices connected by an edge $e_0$, distributing the incident half-edges at $v$ on the new vertices, thus forming a graph $\Gamma'$. Pictorially:
\begin{align*}
    \delta_s^\bullet:
    \begin{tikzpicture}[baseline=-.65ex]
        \node[int] (v) at (0,0) {};
        \draw (v) edge +(-.3,-.3)  edge +(-.3,0) edge +(-.3,.3) edge +(.3,-.3)  edge +(.3,0) edge +(.3,.3);
        \end{tikzpicture}
        &\mapsto\sum
        \begin{tikzpicture}[baseline=-.65ex]
        \node[int] (v) at (0,0) {};
        \node[int] (w) at (0.5,0) {};
        \draw (v) edge (w) (v) edge +(-.3,-.3)  edge +(-.3,0) edge +(-.3,.3)
         (w) edge +(.3,-.3)  edge +(.3,0) edge +(.3,.3);
        \end{tikzpicture}        \quad .
\end{align*}
Similarly, we define $\delta_s^\circ$ by
\[
\delta_s^\circ (\Gamma, e_1\wedge\cdots \wedge e_k \wedge h_1\wedge \cdots \wedge h_w)
=
\sum_{\underset{|B|\geq 2}{B\subset H_*}} (\text{split}_B\Gamma, e_0\wedge e_1\wedge\cdots \wedge e_k \wedge h_1\wedge \cdots \wedge h_w)   \, , 
\]
with the sum running over subsets $B$ of the set $H_*$ of half-edges incident at the special vertex, such that $|B|\geq 2$ and $B$ contains at most one of the distinguished half-edges.
The graph $\text{split}_B\Gamma$ is built by adding a new non-special vertex $v$ to the graph $\Gamma$, with an edge $e_0$ to the special vertex, and reconnecting the half-edges $B$ to $v$.
If $B$ contains a marked half-edge, then the marking is removed and put on $e_0$ instead. 
Pictorially:
\begin{align*}
    \delta_s^\circ:
\begin{tikzpicture}
        \node[ext,accepting] (v) at (0,0) {};
        \draw (v) edge[->-] +(-.3,-.3)  edge[->-] +(-.3,0) edge +(-.3,.3) edge[->-] +(.3,-.3)  edge +(.3,0) edge +(.3,.3);
        \end{tikzpicture}
        &\mapsto\sum
        \begin{tikzpicture}[baseline=-.65ex]
        \node[ext,accepting] (v) at (0,0) {};
        \node[int] (w) at (0.5,0) {};
        \draw (v) edge (w) (v) edge[->-] +(-.3,-.3)  edge[->-] +(-.3,0) edge +(-.3,.3) edge[->-] +(.3,-.3)
         (w)   edge +(.3,0) edge +(.3,.3);
        \end{tikzpicture}
        +\sum 
        \begin{tikzpicture}[baseline=-.65ex]
        \node[ext,accepting] (v) at (0,0) {};
        \node[int] (w) at (0.5,0) {};
        \draw (v) edge[->-] (w) (v) edge[->-] +(-.3,-.3)  edge[->-] +(-.3,0) edge +(-.3,.3)
            (w) edge +(.3,-.3)  edge +(.3,0) edge +(.3,.3);
        \end{tikzpicture} \quad .
\end{align*}

The sum $\d = \d_{\omega} + \d_{s}$ gives a well-defined coboundary operator on $B_{g,n}$ (\cite{PW23}, Lemma 3.4) defining our chain complex.

We conclude this paragraph by noting a useful property of $\d_s^\bullet$ whose proof is immediate.
\begin{proposition}\label{prop:delta_s-distr}
  $\delta_s^\bullet$ distributes over connected components as follows:
  $$\delta_s^\bullet(\bigcup_i C_i) = \sum_i (\delta_s^\bullet(C_i) \cup\bigcup_{j\neq i} C_j)$$
\end{proposition}

\subsubsection{Induced Modules and Young Representations}

To obtain the $\Q$-module generated by a given $\Gamma$ in $B_{g,n}$, one must consider all possible ways of labelling the $n$ legs and the identifications under permutations and isomorphism. This yields a representation of $S^n$ coming from a representation of $S_n$. These are well-known and can be represented using Young diagrams.

As an example, we consider the generators of the form 
$$\Gamma = \oleg{1} \ldots \oleg{11}.$$
Any permutation of two label induces a sign coming from the permutation of the associated distinguished half-edges. Hence, $\Gamma$ generates the sign-representation module $V_{1^{11}}$. 
For a more complex example, one can consider the graphs of the form 
\begin{center}
  \resizebox{0.6\textwidth}{!}{$\Gamma= \oleg{1} \ldots \oleg{10} \tripleo \twotwolegs{11}{$\omega$}{12}{13}$}
\end{center}
These have ten labels that permute with a sign, two that permute without sign and one isolated label that cannot commute without yielding a different graph. The induced representation module can then be computed using Pieri's Rule and is given by
\begin{align*}
   & \rm{ind}_{S_{12} \times S_1} \left( \rm{ind}_{s_{10} \times S_2} \left(V_{1^{10}} \otimes V_2\right) \otimes V_1\right)   
    \\
    &= V_{41^{n-4}} \oplus V_{321^{n-5}} \oplus V_{31^{n-3}} \oplus V_{31^{n-3}} \oplus V_{2^21^{n-4}}  \oplus V_{21^{n-2}}.
\end{align*}
In terms of Young diagrams, this is represented by the following equality:
\begin{center}
  \resizebox{0.9\hsize}{!}{
$\left( \ydiagram{1,1,1,1,1,1,1,1,1,1} \otimes \ydiagram{2}\right) \otimes \ydiagram{1} = \left( \ydiagram{3,1,1,1,1,1,1,1,1,1} \oplus \ydiagram{2,1,1,1,1,1,1,1,1,1,1}\right) \otimes \ydiagram{1}  
=\ydiagram{4,1,1,1,1,1,1,1,1,1} \oplus \ydiagram{3,2,1,1,1,1,1,1,1,1} \oplus \ydiagram{3,1,1,1,1,1,1,1,1,1,1}
\oplus \ydiagram{3,1,1,1,1,1,1,1,1,1,1} \oplus \ydiagram{2,2,1,1,1,1,1,1,1,1,1} \oplus \ydiagram{2,1,1,1,1,1,1,1,1,1,1,1}
$
}
\end{center}
For more details about the underlying theory, we refer the curious reader to the classic book of Fulton and Harris \cite{FH04}

There is another nice consequence of dealing with representation modules spelled out in terms of their decomposition in irreducible representations. By Schur's lemma, the restriction of the differential to an irreducible module (in our case, the module $V_{\lambda}$) is either trivial or an isomorphism onto its image. This simplifies cancellations in the graph complex a lot.

\subsection{Excess}\label{subsec:Excess}

Let us briefly introduce the notion of excess as defined in \cite{PW23}, and motivate its role to describe a notion of complexity for the graph complexes $B_{g,n}$.

\begin{definition}
  For a generating graph $\Gamma$ of the complex $B_{g,n}$ with $w$ distinguished half-edges, we define its \emph{excess} to be the integer
  $$E(\Gamma) = 3 (g-1) + 2n - 2w.$$ 
\end{definition}

In order to work with the blown-up connected components of such graphs, we can define a notion of excess for blown-up parts as well. Given a blown-up component $C_i$ of a graph in $B_{g,n}$, we can define the excess of the component as
$$ E(C_i) = 3 h_1(C_i) + 3e_i + 2 n_i + w_i -3 ,$$
where we decomposed the genus $g_i$ of the component into the sum ${g_i = h^1(C_i) + e_i + w_i}$ of the loop order of $C_i$, the number $e_i$ of non-distinguished and $w_i$ of distinguished half-edges incident to the special vertex.

Note that these definitions give additivity of excess over blown-up components. This means that if $\Gamma = \bigcup_i C_i$ is a graph in $B_{g,n}$ with blown-up connected components $\{C_i\}_i$, then
$$E \left(\bigcup_i C_i\right) = \sum_i E (C_i). $$

We use excess to measure a certain degree of complexity in our graph complexes. To understand this, note that the bound $E(\Gamma) \leq {3g+2n-25}$ is trivially satisfied by any $\Gamma \in B_{g,n}$. As a consequence, a fixed graph complex imposes an upper bound onto its generating graphs.

\begin{definition}
  The \emph{excess} of the graph complex $B_{g,n}$ is given by $E(g,n) := 3g+2n-25$.
\end{definition}

With this nomenclature, we can reformulate our observation above by saying that all generators $\Gamma$ of $B_{g,n}$ have excess of at most the excess of $B_{g,n}$.

For graph complexes of excess four, the generators must then be rather simple in the low excess case as shown by the following result:

\begin{proposition}[\cite{PW23}, 4.4.]
  If $h^1(C_i) \geq 1$, then ${E(C_i) \geq 5}$.
\end{proposition}

Payne and Willwacher computed the cohomology groups of the graph complexes with excess less or equal to three. In the following, we will compute the cohomology in the last remaining case with trivial loop order, excess four. While our approach is mostly computational, we also give additional details about the process and the appearing graphs.

\section{Excess Four Generators}\label{sec:Generators}

In this section, we determine all the generating graphs that play into the cohomology of graph complexes $B_{g,n}$ with excess four. For the remaining of this section, we fix two integers $g,n \geq 0$ with $E(g,n) =4$ and drop the $g,n$ subscript when it eases notation.

As described in Section \ref{sec:Background}, the generators of $B=B_{g,n}$ can be obtained from combining connected components of suitable excess. As excess is non-negative and additive over connected components, it will be sufficient to determine all connected components of excess at most four. As we shall see in Section \ref{subsec:reduction-argt}, we will be able to simplify our computations, by restricting our attention to certain connected components only, which we will call  call \emph{essential} (see Definition \ref{def:essential}). In Section \ref{subsec:genconccomp}, we will then determine all the essential connected components of excess up to four before combining them in all possible ways to yield a generator of our complexes in Section \ref{subsec:GenGraphs}.

  \subsection{The Reduction Argument}\label{subsec:reduction-argt}

  While all Excess four blown-up connected components can be obtained by listing all admissible instances of component $C$ satisfying the equation
  \begin{equation}\label{eq:E-component}
    E(C) = 3 h_i(C) + 3 e +2n + w -3 = 4,
  \end{equation}
  we can simplify computations using the following observation.

  \begin{example}\label{eg:eg-nonnec}
    Consider the connected components of excess four given by the solution of \ref{eq:E-component} given by $e=1$, $n=0$ and $w=4$, which can be represented as follows:

  
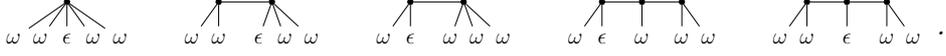
\begin{figure*}[ht]
    \begin{center}
    \begin{tabular}{c c c c c}
      \scalebox{\graphscale}{\fivelegs{$\omega$}{$\omega$}{$\epsilon$}{$\omega$}{$\omega$}} &
      \scalebox{\graphscale}{\twothreelegs{$\omega$}{$\omega$}{$\epsilon$}{$\omega$}{$\omega$}}  &
      \scalebox{\graphscale}{\twothreelegs{$\omega$}{$\epsilon$}{$\omega$}{$\omega$}{$\omega$}} &
      \scalebox{\graphscale}{\twoonetwolegs{$\omega$}{$\epsilon$}{$\omega$}{$\omega$}{$\omega$}}  &
      \scalebox{\graphscale}{\twoonetwolegs{$\omega$}{$\omega$}{$\epsilon$}{$\omega$}{$\omega$}}.
    \end{tabular}
  \end{center}
  \end{figure*}
  
  Note that the last component vanishes in our graph complex because of odd symmetry. The other generators are related via the boundary operator as follows:

  \begin{align*}
    \delta_s^\bullet(\scalebox{\graphscale}{\fivelegs{$\omega$}{$\omega$}{$\epsilon$}{$\omega$}{$\omega$}}) &= 6\cdot \scalebox{\graphscale}{\twothreelegs{$\omega$}{$\omega$}{$\epsilon$}{$\omega$}{$\omega$}} + 4 \cdot \scalebox{\graphscale}{\twothreelegs{$\omega$}{$\epsilon$}{$\omega$}{$\omega$}{$\omega$}} \\
    \delta_s^\bullet(\scalebox{\graphscale}{\twothreelegs{$\omega$}{$\omega$}{$\epsilon$}{$\omega$}{$\omega$}}) &= 2 \cdot \scalebox{\graphscale}{\twoonetwolegs{$\omega$}{$\epsilon$}{$\omega$}{$\omega$}{$\omega$}}   \\
    \delta_s^\bullet(\scalebox{\graphscale}{\twothreelegs{$\omega$}{$\epsilon$}{$\omega$}{$\omega$}{$\omega$}}) &= 3 \cdot \scalebox{\graphscale}{\twoonetwolegs{$\omega$}{$\epsilon$}{$\omega$}{$\omega$}{$\omega$}} \\
    \delta_s^\bullet(\scalebox{\graphscale}{\twoonetwolegs{$\omega$}{$\epsilon$}{$\omega$}{$\omega$}{$\omega$}}) &=0
  \end{align*}

  As we will see, these relations make the impact of such components cancel out in homology, whence we will omit them in our computations.
  \end{example}

  We now give the reduction argument.

  Let $\cS$ be the set of excess four connected components with $w \geq 4$. Additionally to the connected components considered in Example \ref{eg:eg-nonnec}, this amounts to components with $w=5$, $n=1$ and to those with $w=7$.

  \begin{definition}
    We denote by $K=K_{g,n}$ the subset $K \subseteq B$ of graphs, which have a component from $\cS$ in their blown-up representation.
  \end{definition}

  \begin{remark}\label{rmk:decompK}
    Recall that excess is additive over connected components and that $E(\Gamma) \leq E(g,n) = 4$ holds for all $\Gamma \in B$. Therefore, and since all components in $\cS$ have excess four, any graph $\Gamma \in K$ can be written as 
    $$\Gamma = \Gamma_0 \cup C_{\Gamma}$$
    in blown-up representation, where $C_\Gamma \in \cS$ and $\Gamma_0$ is a union of excess zero connected components.
  \end{remark}

  \begin{proposition}
    $(K,\delta)$ forms a subcomplex of $(B,\delta)$
  \end{proposition}
  \begin{proof}
    We check that $\delta(K) \subseteq K$. Clearly, $\delta_s^{\bullet}(K)\subseteq K$ by construction and $\delta_\omega$ is trivial as it increases excess by two. Hence, we only need to verify that $\delta_s^{\circ}(K) \subseteq K$, i.e. that $\delta_s^{\circ}$ preserves the $C_\Gamma \in \cS$ component of any $\Gamma \in K$. 
    
    The only way of altering a $C_\Gamma$ is to combine it with one or more excess zero components. However, the latter do not have any unmarked legs. Hence, by its definition, $\delta_s^{\circ}$ can only combine a component $C_\Gamma \in K$ with $w=4$ together with an excess zero component, thus yielding one of the two other types of components in $K$, namely $w=5$ or $w=7$. This shows that $\delta_s^{\circ}$ preserves $K$ and completes the proof.
  \end{proof}

  The graphs in $K$ do not play into the cohomology of $B$ as shows the following result.

  \begin{proposition}\label{prop:notessential}
    Let $\hat B_{g,n} = B_{g,n} /K_{g,n}$. Then the quotient map $\pi \colon B_{g,n} \to \hat B_{g,n}$ induces a isomorphism on cohomology.
  \end{proposition}
  \begin{proof}
    
  We filtrate $B$ by the number of connected components. Formally, we denote $ (B^k)^d$ the degree $d$ graphs in $B$ with at most $k$ connected components in their blown-up representation and consider the filtration
  $$ (B^1)^\bullet \subseteq(B^2)^\bullet \subseteq(B^3)^\bullet \subseteq(B^4)^\bullet \subseteq \cdots$$
  This filtration induces a cohomology spectral sequence by setting $E_{k,q}^0 = (B^k)^{k+q} /(B^{k-1})^{k+q}$ and $E_{k,q}^1 = H^{k+q}(E_{k\bullet}^0)$ for all $k,q \geq 0$ (see for example \cite{Wei94} theorem 5.4.1).

  Analogously, we can filtrate the quotient $\hat B$ and obtain a spectral sequence $\hat E_{k,q}^r$

  We now consider the quotient chain map
  $$ \pi \colon B \to \hat{B} , $$
  and the induced morphism of spectral sequence
  $$ \pi_{k,q}^r \colon E_{k,q}^r \to \hat{E}_{k,q}^r.$$

  \begin{claim}\label{clm:pi-iso}
    $\pi_{k,q}^1$ defines an isomorphism for all $k,q$.
  \end{claim}
  \begin{proof}[Proof of Claim \ref{clm:pi-iso}]
    Consider the short exact sequence on the chain level given by
  $$0 \to K^\bullet \overset{\iota}{\longrightarrow} B^{\bullet} \overset{\pi}{\longrightarrow} \hat B^\bullet \to 0 $$
  inducing on the filtration the short exact sequence
  $$0 \to (K^k)^\bullet / (K^{k-1})^\bullet   (B^k)^\bullet / (B^{k-1})^\bullet  \to (\hat B^k)^\bullet / (\hat B^{k-1})^\bullet  \to 0.$$
  This induces on the cohomology level a long exact sequence of the form
  $$ \cdots \to \hat E_{k,q}^1 \to E_{k,q}^1 \to H^{k+q}(K^k / K^{k-1}) \to \hat E_{k,q+1} \to \cdots $$
  for all $k$. Using this exact sequence, it is now sufficient to verify that $H^{d}(K^k / K^{k-1}) = 0$ for all $k, q \geq 0$ to prove our claim, which we shall now prove.

  Fix $k\geq 0$. As in Remark \ref{rmk:decompK}, we note that any $\Gamma \in K^k/K^{k-1}$ can be written as $\Gamma = \Gamma_0 \cup \C_{\Gamma}$ in blown-up decomposition, where $C_\Gamma \in \cS$ and $\Gamma_0$ is a union of $k-1$ excess zero connected components. Note that $\delta_\omega$ and $\delta_s^{\circ}$ are trivial because of a simple excess argument and the grading respectively. We therefore only need to consider the $\delta_s^{\bullet}$ part of the coboundary operator in this computation. Finally, we note that we can compute the cohomology of $K^k / K^{k-1}$ via a spectral sequence comparing on the first pages only graphs with same $\Gamma_0$. Then the cancellation within $\cS$ leads to all homology groups being trivial on the first page of the spectral sequence already, whence $H^{d}(K^k / K^{k-1}) = 0$ for all $q \geq 0$.
  \end{proof}

  Using the comparison theorem for spectral sequences (see for instance \cite{Wei94}, theorem 5.2.12), since $\pi_{k,q}^1$ is an isomorphism for every $k,q$, we obtain that $\pi$ induces an isomorphism of cohomology
  $$\pi^* \colon H^\bullet(\hat B) \overset{\cong}{\longrightarrow} H^\bullet( B)$$
  as desired.
  \end{proof}

  \begin{definition}\label{def:essential}
    We will say that the connected components listed in Proposition \ref{prop:notessential} are not \emph{essential} to the cohomology of $B_{g,n}$.
  \end{definition}

  In order to simplify the computations, we will therefore restrain our attention to essential connected components and the graphs built out of these. As an abuse of notation, we shall write $B_{g,n}$ instead of $\hat B_{g,n}$ when no ambiguity can arise.

  \subsection{Generating Connected Components} \label{subsec:genconccomp}

  As the main result of this section, we determine the connected components of excess four that will be necessary for the generators of the cohomology of $B_{g,n}$
  \begin{proposition}\label{prop:concgen4}
    The essential generating components of excess four are the components of the following form:
  
    \begin{center}
      \fourlegs{i}{$\omega$}{$\omega$}{$\epsilon$}
      \quad
      \twotwolegs{$i$}{$\omega$}{$\omega$}{$\epsilon$}
      \quad
      \twoonetwolegs{i}{$\omega$}{$\omega$}{$\omega$}{j}
      \quad
      \fourlegs{i}{$\omega$}{j}{k}
      \\
      \twotwolegs{i}{$\epsilon$}{$\omega$}{$\omega$}
      \quad
      \threelegs{$\epsilon$}{i}{j}
      \quad
      \twotwolegs{i}{$\omega$}{j}{k}
    \end{center}
    \end{proposition}

  This extends the result of Payne-Willwacher, who computed the essential connected components of excess up to $3$. These components are summarized in the table below.
  \begin{figure}[ht]\label{fig:concgen0123}
  \begin{center}
    \begin{tabular}{|c  c | c  c  c |}
      \hline
      \multicolumn{2}{|c|}{$E=0$} &
      \multicolumn{3}{|c|}{$E=2$}
      \\
      
      \oleg{j} & \tripleo &
      \eleg{j} & \threelegs{$\epsilon$}{$\omega$}{$\omega$} & \threelegs{i}{j}{$\omega$}
      \\
      \hline
      \multicolumn{2}{|c|}{$E=1$} &
      \multicolumn{3}{|c|}{$E=3$}
      \\
      
      \oleg{$\epsilon$} & \threelegs{j}{$\omega$}{$\omega$} &
      \eleg{$\epsilon$} & \threelegs{$\epsilon$}{j}{$\omega$} & \twotwolegs{i}{$\omega$}{$\omega$}{j} 
      \\
      \hline
    \end{tabular}
  \end{center}
  \caption{Essential Connected Components with Excess $\leq 3$ }
  \end{figure}

  \begin{proof}[Proof of Proposition \ref{prop:concgen4}]
  We check all combinatoric solutions to the equation \ref{eq:E-component}.
  In this low excess case, generators are known to satisfy ${h_i(C) = 0}$ (see \cite{PW23}, Lemma 4.4).
  
  Observe that as mentioned above, the component $\scalebox{\graphscale}{\threelegs{$\epsilon$}{$\omega$}{$\epsilon$}} $ vanishes in our graph complex because of the odd symmetry of the two $\epsilon$-legs.
  Further, considering the simplifications from Proposition \ref{prop:notessential}, we may restrict to the essential connected components listed Proposition \ref{prop:concgen4}.
  \end{proof} 

  \begin{remark}\label{rmk:expgrowth}
    Note that the number of generators grows exponentially with excess. This is a direct consequence of the additivity of excess on connected components. As a consequence, the complexity of explicit computation of the cohomology of the $B_{g,n}$ grows exponentially with excess.
  \end{remark}

\subsection{Generators for Cohomology}\label{subsec:GenGraphs}

We now assemble the obtained essential connected components in all possible ways to obtain generators for $B_{g,n}$. Since the excess is additive over the blown-up connected components, any such combination will determine a partition $\lambda$ of four. We classify the generators by the partition they determine. 

\vskip \fsize
\bm{$\lambda= 4.$}
\hfill\\
\resizebox{\textwidth}{!}{
  \begin{tabular}{l | c | c}
    Degree & Generator & Module \\
    \hline
    $\frac{3}{2} (g-1) + 12$ &
    $\Gamma^{(4)}_{\epsilon ij } = \oleg{1} \ldots \oleg{n} \tripleo\ldots\tripleo \threelegs{$\epsilon$}{i}{j}$ &
    $V_{31^{n-3}} \oplus V_{21^{n-2}}$ \\
    $\frac{3}{2} (g-1) + 12$ &
    $\Gamma^{(4)}_{ ijk} = \oleg{1} \ldots \oleg{n} \tripleo\ldots\tripleo  \fourlegs{i}{$\omega$}{j}{k}$ &
    $V_{41^{n-4}} \oplus V_{31^{n-3}}$ \\
    $\frac{3}{2} (g-1) + 13$ &
    $\Gamma^{(4)}_{i\cdot jk } = \oleg{1} \ldots \oleg{n} \tripleo\ldots\tripleo  \twotwolegs{i}{$\omega$}{j}{k}$ &
    $V_{41^{n-4}} \oplus V_{321^{n-5}} \oplus V_{31^{n-3}} \oplus V_{2^21^{n-4}} \oplus V_{31^{n-3}} \oplus V_{21^{n-2}}$ \\
    $\frac{3}{2} (g-3) + 14$ &
    $\Gamma^{(4)}_{\epsilon  i} = \oleg{1} \ldots \oleg{n} \tripleo\ldots\tripleo \fourlegs{i}{$\omega$}{$\omega$}{$\epsilon$}$ & $V_{21^{n-2}} \oplus V_{1^n}$ \\
    $\frac{3}{2} (g-3) + 15$ &
    $\Gamma^{(4)}_{\epsilon i \cdot \omega} = \oleg{1} \ldots \oleg{n} \tripleo\ldots\tripleo \twotwolegs{$\epsilon$}{i}{$\omega$}{$\omega$}$ & $V_{21^{n-2}} \oplus V_{1^n}$ \\
    $\frac{3}{2} (g-3) + 15$ &
    $\Gamma^{(4)}_{\epsilon \cdot i} = \oleg{1} \ldots \oleg{n} \tripleo\ldots\tripleo  \twotwolegs{i}{$\omega$}{$\omega$}{$\epsilon$}$ &
    $V_{21^{n-2}} \oplus V_{1^n}$ \\
    $ \frac{3}{2} (g-3) + 16$ &
    $\Gamma^{(4)}_{ i\omega j} = \oleg{1} \ldots \oleg{n} \tripleo\ldots\tripleo  \twoonetwolegs{i}{$\omega$}{$\omega$}{$\omega$}{j}$ &
    $V_{2^21^{n-4}} \oplus V_{21^{n-2}} \oplus V_{1^n}$ \\
  \end{tabular}
}

\vskip \fsize
\bm{$\lambda= 31.$}
\hfill\\
\resizebox{\textwidth}{!}{
  \begin{tabular}{l | c | c}
    Degree & Generator & Module \\
    \hline
    $\frac{3}{2} (g-3) + 13$ &
    $\Gamma^{(31)}_{ \epsilon\epsilon;\epsilon} = \oleg{1} \ldots \oleg{n} \tripleo\ldots\tripleo \oleg{$\epsilon$} \eleg{$\epsilon$}$ &
    $V_{1^n}$ \\
    $\frac{3}{2} (g-3) + 14$ &
    $\Gamma^{(31)}_{ \epsilon\epsilon;i} = \oleg{1} \ldots \oleg{n} \tripleo\ldots\tripleo \threelegs{i}{$\omega$}{$\omega$} \eleg{$\epsilon$}$ &
    $V_{21^{n-2}} \oplus V_{1^n}$ \\
    $\frac{3}{2} (g-3) + 14$ &
    $\Gamma^{(31)}_{\epsilon\omega i;\epsilon } = \oleg{1} \ldots \oleg{n} \tripleo\ldots\tripleo \oleg{$\epsilon$} \threelegs{$\epsilon$}{i}{$\omega$}$ &
    $V_{21^{n-2}} \oplus V_{1^n}$ \\
    $\frac{3}{2} (g-3) + 15$ &
    $\Gamma^{(31)}_{\epsilon\omega i ; j } = \oleg{1} \ldots \oleg{n} \tripleo\ldots\tripleo \threelegs{j}{$\omega$}{$\omega$} \threelegs{$\epsilon$}{i}{$\omega$}$ &
    $V_{31^{n-3}} \oplus V_{2^21^{n-4}} \oplus V_{21^{n-2}} \oplus V_{21^{n-2}} \oplus V_{1^n}$ \\
    $\frac{3}{2} (g-3) + 15$ &
    $\Gamma^{(31)}_{ i\cdot j;\epsilon } = \oleg{1} \ldots \oleg{n} \tripleo\ldots\tripleo \oleg{$\epsilon$} \twotwolegs{j}{$\omega$}{$\omega$}{i}$ &
    $V_{31^{n-3}} \oplus V_{21^{n-2}}$ \\
    $\frac{3}{2} (g-3) + 16$ &
    $\Gamma^{(31)}_{i\cdot j ; k } = \oleg{1} \ldots \oleg{n} \tripleo\ldots\tripleo \threelegs{k}{$\omega$}{$\omega$} \twotwolegs{j}{$\omega$}{$\omega$}{i}$ &
    $V_{41^{n-4}} \oplus V_{321^{n-5}} \oplus V_{31^{n-3}} \oplus V_{2^21^{n-4}} \oplus V_{31^{n-3}} \oplus V_{21^{n-2}}$ 
  \end{tabular}
}

\vskip \fsize
\bm{$\lambda= 2^2.$}
\hfill\\
\resizebox{\textwidth}{!}{
  \begin{tabular}{l | c | c}
    Degree & Generator & Module \\
    \hline
    $\frac{3}{2} (g-1) + 11$ &
    $\Gamma^{(2^2)}_{\epsilon i ;\epsilon j } = \oleg{$1$} \ldots \oleg{n} \tripleo\ldots\tripleo \eleg{i} \eleg{j}$ &
    $V_{31^{n-3}} \oplus V_{21^{n-2}}$ \\
    $ \frac{3}{2} (g-1) + 12$ &
    $\Gamma^{(2^2)}_{ \epsilon i ;jk} = \oleg{$1$} \ldots \oleg{n} \tripleo\ldots\tripleo \eleg{i} \threelegs{j}{k}{$\omega$}$ &
    $V_{41^{n-4}} \oplus V_{321^{n-5}} \oplus V_{31^{n-3}} \oplus V_{2^21^{n-4}} \oplus V_{31^{n-3}} \oplus V_{21^{n-2}}$  \\
    $ \frac{3}{2} (g-3) + 14$ &
    $\Gamma^{(2^2)}_{ \epsilon i ; \epsilon } = \oleg{$1$} \ldots \oleg{n} \tripleo\ldots\tripleo \eleg{i} \threelegs{$\epsilon$}{$\omega$}{$\omega$}$ &
    $V_{21^{n-2}} \oplus V_{1^n}$ \\
    $ \frac{3}{2} (g-1) + 13$ &
    $\Gamma^{(2^2)}_{ ij;kl} = \oleg{$1$} \ldots \oleg{n} \tripleo\ldots\tripleo \threelegs{i}{j}{$\omega$} \threelegs{k}{l}{$\omega$}$ &
    $V_{51^{n-5}} \oplus V_{41^n-4} \oplus V_{3^21^{n-6}} \oplus V_{321^{n-5}} \oplus V_{2^21^{n-4}}$\footnotemark{} \\
    $ \frac{3}{2} (g-3) + 15$ &
    $\Gamma^{(2^2)}_{ \epsilon;ij} = \oleg{$1$} \ldots \oleg{n} \tripleo\ldots\tripleo \threelegs{$\epsilon$}{$\omega$}{$\omega$} \threelegs{i}{j}{$\omega$}$ &
    $V_{31^{n-3}} \oplus V_{21^{n-2}}$ \\
    $ \frac{3}{2} (g-5) + 17$ &
    $\Gamma^{(2^2)}_{ \epsilon;\epsilon} = \oleg{$1$} \ldots \oleg{n} \tripleo\ldots\tripleo \threelegs{$\epsilon$}{$\omega$}{$\omega$} \threelegs{$\epsilon$}{$\omega$}{$\omega$}$ &
    $V_{1^n}$ 
  \end{tabular}
}
\footnotetext{Here, one needs to pay attention to the "block symmetry" allowing to permute the pairs $\{i,j\}$ and $\{k,l\}$ without sign, whence the induced module is the 3-dimensional $V_4 \oplus V_{2^2}$.}

\vskip \fsize
\bm{$\lambda= 21^2.$}
\hfill\\
\resizebox{\textwidth}{!}{
  \begin{tabular}{l | c | c}
    Degree & Generator & Module \\
    \hline
    $\frac{3}{2} (g-3) + 13$ &
    $\Gamma^{(21^2)}_{ \epsilon i} = \oleg{$1$} \ldots \oleg{n} \tripleo\ldots\tripleo \oleg{$\epsilon$} \oleg{$\epsilon$} \eleg{i}$ &
    $V_{21^{n-2}} \oplus V_{1^n}$ \\
    $\frac{3}{2} (g-3) + 14$ &
    $\Gamma^{(21^2)}_{ \epsilon i ;j} = \oleg{$1$} \ldots \oleg{n} \tripleo\ldots\tripleo \oleg{$\epsilon$} \threelegs{j}{$\omega$}{$\omega$} \eleg{i}$ &
    $V_{31^{n-3}} \oplus V_{2^21^{n-4}} \oplus V_{21^{n-2}} \oplus V_{21^{n-2}} \oplus V_{1^n}$ \\
    $\frac{3}{2} (g-3) + 15$ &
   $ \Gamma^{(21^2)}_{\epsilon i ;j;k } = \oleg{$1$} \ldots \oleg{n} \tripleo\ldots\tripleo \threelegs{k}{$\omega$}{$\omega$} \threelegs{j}{$\omega$}{$\omega$} \eleg{i}$ &
   $V_{41^{n-4}} \oplus V_{321^{n-5}} \oplus V_{31^{n-3}} \oplus V_{2^21^{n-4}} \oplus V_{31^{n-3}} \oplus V_{21^{n-2}}$ \\
   $\frac{3}{2} (g-3) + 14$ &
   $\Gamma^{(21^2)}_{ij} = \oleg{$1$} \ldots \oleg{n} \tripleo\ldots\tripleo \oleg{$\epsilon$} \oleg{$\epsilon$} \threelegs{j}{i}{$\omega$}$ &
   $V_{31^{n-3}} \oplus V_{21^{n-2}}$ \\
   $\frac{3}{2} (g-3) + 15$ &
   $\Gamma^{(21^2)}_{ ij;k} = \oleg{$1$} \ldots \oleg{n} \tripleo\ldots\tripleo \oleg{$\epsilon$} \threelegs{k}{$\omega$}{$\omega$} \threelegs{j}{i}{$\omega$}$ &
   $V_{41^{n-4}} \oplus V_{321^{n-5}} \oplus V_{31^{n-3}} \oplus V_{2^21^{n-4}} \oplus V_{31^{n-3}} \oplus V_{21^{n-2}}$\\
   $\frac{3}{2} (g-3) + 16$ &
   $\Gamma^{(21^2)}_{ ij;k;l} = \oleg{$1$} \ldots \oleg{n} \tripleo\ldots\tripleo \threelegs{l}{$\omega$}{$\omega$} \threelegs{k}{$\omega$}{$\omega$} \threelegs{j}{i}{$\omega$}$ &
   $ V_{51^{n-5}} \oplus V_{421^{n-6}} \oplus V_{41^{n-4}} \oplus V_{3^21^{n-6}} \oplus V_{321^{n-5}} \oplus V_{41^{n-4}} \oplus V_{321^{n-5}} \oplus V_{31^{n-3}} \oplus V_{2^21^{n-4}} $ \\
   $\frac{3}{2} (g-5) + 16$ &
   $\Gamma^{(21^2)}_{ \epsilon} = \oleg{$1$} \ldots \oleg{n} \tripleo\ldots\tripleo \oleg{$\epsilon$} \oleg{$\epsilon$} \threelegs{$\epsilon$}{$\omega$}{$\omega$}$ &
   $V_{1^{n}}$ \\
   $\frac{3}{2} (g-5) + 17$ &
   $\Gamma^{(21^2)}_{ \epsilon;i} = \oleg{$1$} \ldots \oleg{n} \tripleo\ldots\tripleo \oleg{$\epsilon$} \threelegs{i}{$\omega$}{$\omega$} \threelegs{$\epsilon$}{$\omega$}{$\omega$}$ &
   $V_{21^{n-2}} \oplus V_{1^n}$ \\
   $\frac{3}{2} (g-5) + 18$ &
   $\Gamma^{(21^2)}_{ \epsilon;i;j} =\oleg{$1$} \ldots \oleg{n} \tripleo\ldots\tripleo \threelegs{i}{$\omega$}{$\omega$} \threelegs{j}{$\omega$}{$\omega$} \threelegs{$\epsilon$}{$\omega$}{$\omega$}$ &
   $V_{31^{n-3}} \oplus V_{21^{n-2}}$
  \end{tabular}
}

\vskip \fsize
\bm{$\lambda= 1^4.$}
\hfill\\
\resizebox{\textwidth}{!}{
  \begin{tabular}{l | c | c}
    Degree & Generator & Module \\
    \hline
    $\frac{3}{2} (g-5) + 15$ &
    $\Gamma^{(1^4)}_{\omega\epsilon} =
    \oleg{$1$} \ldots \oleg{n} 
    \tripleo\ldots\tripleo
    \oleg{$\epsilon$}
    \oleg{$\epsilon$} 
    \oleg{$\epsilon$}
    \oleg{$\epsilon$}$ &
    $ V_{1^n}$ \\
    $\frac{3}{2} (g-5) + 16$ &
    $\Gamma^{(1^4)}_{i} =
    \oleg{$1$} \ldots \oleg{n}
    \tripleo\ldots\tripleo 
    \oleg{$\epsilon$}
    \oleg{$\epsilon$} 
    \oleg{$\epsilon$}
    \threelegs{i}{$\omega$}{$\omega$}$ &
    $ V_{21^{n-2}} \oplus V_{1^n}$  \\
    $\frac{3}{2} (g-5) + 17$ &
    $\Gamma^{(1^4)}_{i;j} =
    \oleg{$1$} \ldots \oleg{n} 
    \tripleo\ldots\tripleo
    \oleg{$\epsilon$}
    \oleg{$\epsilon$} 
    \threelegs{i}{$\omega$}{$\omega$} 
    \threelegs{j}{$\omega$}{$\omega$}$ &
    $ V_{31^{n-3}} \oplus V_{21^{n-2}} $ \\
    $\frac{3}{2} (g-5) + 18$ &
    $\Gamma^{(1^4)}_{i;j;k} =
    \oleg{$1$} \ldots \oleg{n} 
    \tripleo\ldots\tripleo
    \oleg{$\epsilon$}
    \threelegs{i}{$\omega$}{$\omega$} 
    \threelegs{j}{$\omega$}{$\omega$} 
    \threelegs{k}{$\omega$}{$\omega$}$ &
    $ V_{41^{n-4}} \oplus V_{31^{n-3}} $ \\
    $\frac{3}{2} (g-5) + 19$ &
    $\Gamma^{(1^4)}_{i;j;k;l} =
    \oleg{$1$} \ldots \oleg{n} 
    \tripleo\ldots\tripleo
    \threelegs{i}{$\omega$}{$\omega$} 
    \threelegs{j}{$\omega$}{$\omega$} 
    \threelegs{k}{$\omega$}{$\omega$} 
    \threelegs{l}{$\omega$}{$\omega$}$ &
    $ V_{51^{n-5}} \oplus V_{41^{n-4}} $ 
  \end{tabular}
}
\vskip \fsize

In general, the boundary operator on our graph complex does not preserve excess, but it does preserve its parity. This comes from the ``$-2w$'' term in the definition of excess which is not preserved by $\d_{\omega}$. Therefore, in addition to the listed generators of excess four, we will also need to consider the following generators of excess two and zero.

\vskip \fsize
\resizebox{0.9\textwidth}{!}{
  \begin{tabular}{l | c | c}
    Degree & Generator & Module \\
    \hline
    $\frac{3}{2} (g-3) + 13$ &
    $\Gamma^{(2)}_{\epsilon} = \oleg{1} \ldots \oleg{n} \tripleo \ldots \tripleo \threelegs{$\epsilon$}{$\omega$}{$\omega$}$&
    $V_{1^n}$ \\
    $\frac{3}{2} (g-1) + 10$ &
    $\Gamma^{(2)}_{\epsilon i} =  \oleg{1} \ldots \oleg{n} \tripleo \ldots \tripleo \eleg{i}$ &
    $ V_{21^{n-2}} \oplus V_{1^n}$ \\
    $\frac{3}{2} (g-1) + 11$ &
    $\Gamma^{(2)}_{ij} = \oleg{1} \ldots \oleg{n} \tripleo \ldots \tripleo \threelegs{i}{j}{$\omega$}$ &
    $ V_{31^{n-3}} \oplus V_{21^{n-2}}$ \\
    $\frac{3}{2} (g-3) + 12$ &
    $\Gamma^{(2)}_{\omega\epsilon\omega\epsilon} =\oleg{1} \ldots \oleg{n} \tripleo \ldots \tripleo \oleg{$\epsilon$} \oleg{$\epsilon$}$ &
    $ V_{1^n}$ \\
    $\frac{3}{2} (g-3) + 13$ &
    $\Gamma^{(2)}_{\omega\epsilon i} = \oleg{1} \ldots \oleg{n} \tripleo \ldots \tripleo \oleg{$\epsilon$} \threelegs{i}{$\omega$}{$\omega$}$ &
    $ V_{21^{n-2}} \oplus V_{1^n}$ \\
    $\frac{3}{2} (g-3) + 14$ &
    $\Gamma^{(2)}_{i;j} = \oleg{1} \ldots \oleg{n} \tripleo \ldots \tripleo \threelegs{i}{$\omega$}{$\omega$} \threelegs{j}{$\omega$}{$\omega$}$ &
    $ V_{31^{n-3}} \oplus  V_{21^{n-2}}$\\
    %
    %
    $\frac{3}{2} (g-1) + 9$ &
    $\Gamma^{(0)} = \oleg{1} \ldots \oleg{n} \tripleo \ldots \tripleo$ &
    $ V_{1^{n}} $
  \end{tabular}
}

\section{Computing Cohomology}\label{sec:Cohomology}

In this section, we compute the Cohomology groups of the excess four graph complexes, thus proving Theorem \ref{thm:cohomologyBgn}. 

To do so, we recall that any generator of a graph complex $B_{g,n}$ of excess four we consider is assembled of essential connected components whose excess adds up to four, two or zero. Hence, any graph comes with an associated partition $\lambda$ of four. We consider the filtration on $B_{g,n}$ induced by the lexicographic order on the partitions of four and study the corresponding spectral sequence. 

\begin{proof}[Proof of Theorem \ref{thm:cohomologyBgn}]

For each fixed partition, we first study the cohomology groups of the obtained subcomplex of $B_{g,n}$ generated only by the generators corresponding to the chosen partition. Note that this corresponds to the first page of our spectral sequence.

We include the impact of generators of Excess zero and two on cancellation on later pages of the spectral sequence right away when it makes the writing more concise.

\vskip \fsize
\bm{$\lambda= 4.$}
The generators are determined by their excess four component. The chain of generators with component \scalebox{\graphscale}{\threelegs{$\epsilon$}{i}{j}}, with \scalebox{\graphscale}{$\fourlegs{$\omega$}{i}{j}{k}$}, and with \scalebox{\graphscale}{$\twotwolegs{$\omega$}{i}{j}{k}$} contributes to the cohomology on the first page of the spectral sequence with a $V_{321^{n-5}} \oplus V_{2^21^{n-4}}$ term in degree $ d = \frac{3}{2} (g-1) + 13$.

The graphs with component \scalebox{\graphscale}{\fourlegs{i}{$\omega$}{$\omega$}{$\epsilon$}}, 
\scalebox{\graphscale}{\twotwolegs{i}{$\epsilon$}{$\omega$}{$\omega$}}, 
\scalebox{\graphscale}{\twotwolegs{i}{$\omega$}{$\omega$}{$\epsilon$}} 
and \scalebox{\graphscale}{\twoonetwolegs{i}{$\omega$}{$\omega$}{$\omega$}{j}} contribute together through a $V_{2^21^{n-4}}$ term in degree $ d = \frac{3}{2} (g-3) + 16$  when $n>1$. When $n=1$, the contribution is solely a $V_1$ term in degree $ d = \frac{3}{2} (g-3) + 15$, induced by the class of the generator with the \scalebox{\graphscale}{\twotwolegs{i}{$\omega$}{$\omega$}{$\epsilon$}} component.

\vskip \fsize
\bm{$\lambda= 31.$ }
The generators with a single excess three and one excess one connected component all have genus of at least three, coming from their excess three part. When $n>1$, they contribute to a $V_{41^{n-4}} \oplus V_{321^{n-5}}$ term in cohomology in degree $ d = \frac{3}{2} (g-3) + 16$ represented by the graph
\begin{center}
  \resizebox{0.6\hsize}{!}{
  $\oleg{1} \ldots \oleg{n} \tripleo\ldots\tripleo \threelegs{k}{$\omega$}{$\omega$} \twotwolegs{j}{$\omega$}{$\omega$}{i}.$
}
\end{center}
When $n=1$, the only contribution is a $V_1$ term in degree $ d = {\frac{3}{2} (g-3) + 14} $ represented by
\begin{center}
  \resizebox{0.65\hsize}{!}{
    $\Gamma^{(31)}_{ \epsilon\epsilon;i} = \oleg{1} \ldots \oleg{n} \tripleo\ldots\tripleo \threelegs{i}{$\omega$}{$\omega$} \eleg{$\epsilon$}
    . $
  }
\end{center}

\vskip \fsize
\bm{$\lambda= 2^2.$ }
The generators without any \scalebox{\graphscale}{$\threelegs{$\epsilon$}{$\omega$}{$\omega$}$} component, along with the lower excess graphs $\Gamma^{(2)}_{\epsilon i}$ and $\Gamma^{(0)}$ (considering cancellation on later pages of the spectral sequence), contribute to a $V_{31^{n-3}} \oplus V_{21^{n-2}}$ term in degree $ d = \frac{3}{2} (g-1) + 12$ represented by
\begin{center}
  \resizebox{0.65\hsize}{!}{
    $\Gamma^{(2^2)}_{ \epsilon k ;ji} = \oleg{$1$} \ldots \oleg{n} \tripleo\ldots\tripleo \eleg{k} \threelegs{j}{i}{$\omega$} $
  }
\end{center}
and a $V_{51^{n-5}} \oplus V_{3^21^{n-6}}$ term in degree $ d = \frac{3}{2} (g-1) + 13$ represented by
\begin{center}
  \resizebox{0.6\hsize}{!}{
    $ \oleg{$1$} \ldots \oleg{n} \tripleo\ldots\tripleo \threelegs{i}{j}{$\omega$} \threelegs{k}{l}{$\omega$}
    . $
  }
\end{center}

Note that this only affects complexes with $n>1$. For $n=1$, this part of the complex cancels out on the first page of the spectral sequence.

The generators with precisely one \scalebox{\graphscale}{$\threelegs{$\epsilon$}{$\omega$}{$\omega$}$} component together with the graph $\Gamma^{(2)}_{\epsilon}$ contribute to a $V_{31^{n-3}}$ term in degree $ d = \frac{3}{2} (g-3) + 15$ represented by
\begin{center}
  \resizebox{0.6\hsize}{!}{
    $ \oleg{$1$} \ldots \oleg{n} \tripleo\ldots\tripleo \threelegs{$\epsilon$}{$\omega$}{$\omega$} \threelegs{i}{j}{$\omega$}
  . $
  }
\end{center}

Once again, there is no impact in the $n=1$ case.

Finally, the generator
\begin{center}
  \resizebox{0.65\hsize}{!}{
    $\Gamma^{(2^2)}_{ \epsilon;\epsilon} = \oleg{$1$} \ldots \oleg{n} \tripleo\ldots\tripleo \threelegs{$\epsilon$}{$\omega$}{$\omega$} \threelegs{$\epsilon$}{$\omega$}{$\omega$} $
  }
\end{center}
contributes to a $V_{1^n}$ term in degree $ d = \frac{3}{2} (g-5) + 17$.

\vskip \fsize
\bm{$\lambda= 21^2.$ }
We split according to the excess two component. The generators with excess two component \scalebox{\graphscale}{$\threelegs{$\epsilon$}{$\omega$}{$\omega$}$} contribute to a $V_{31^{n-3}} $ term in degree $d = \frac{3}{2} (g-5) + 18$ represented by
\begin{center}
  \resizebox{0.68\hsize}{!}{
    $\Gamma^{(21^2)}_{ \epsilon;i;j} =\oleg{$1$} \ldots \oleg{n} \tripleo\ldots\tripleo \threelegs{i}{$\omega$}{$\omega$} \threelegs{j}{$\omega$}{$\omega$} \threelegs{$\epsilon$}{$\omega$}{$\omega$} $
  }
\end{center}
for $n>1$. When $n=1$, this complex does not affect cohomology.

If $n> 1$, the complex given by the graphs with excess two component \scalebox{\graphscale}{\eleg{i}} or \scalebox{\graphscale}{\threelegs{i}{j}{$\omega$}} contributes through a $V_{31^{n-3}}$ term in degree $ {d = \frac{3}{2} (g-3) + 15}$ represented by
\begin{center}
  \resizebox{0.65\hsize}{!}{
    $\oleg{$1$} \ldots \oleg{n} \tripleo\ldots\tripleo \threelegs{k}{$\omega$}{$\omega$} \threelegs{j}{$\omega$}{$\omega$} \eleg{i} $
  }
\end{center}
and to $V_{51^{n-5}} \oplus V_{421^{n-6}} \oplus V_{3^21^{n-6}}$ in degree $ d = \frac{3}{2} (g-3) + 16$. In the case $n=1$, the contribution comes from the only present graph of the complex
\begin{center}
  \resizebox{0.55\hsize}{!}{
    $\Gamma^{(21^2)}_{ \epsilon i} = \tripleo\ldots\tripleo \oleg{$\epsilon$} \oleg{$\epsilon$} \eleg{1} ,$
  }
\end{center}
which contributes a $V_{1}$ term in degree $ d = \frac{3}{2} (g-3) + 13$.

\vskip \fsize
\bm{$\lambda= 1^4.$ }
The generators with four excess one connected components all have genus $g\geq 5$. On the first page of the spectral sequence, this results in a $V_{51^{n-5}}$ term in degree $ d = \frac{3}{2} (g-5) + 19$ represented by
\begin{center}
  \resizebox{0.65\hsize}{!}{
    $
  \oleg{$1$} \ldots \oleg{n} 
  \tripleo\ldots\tripleo
  \threelegs{i}{$\omega$}{$\omega$} 
  \threelegs{j}{$\omega$}{$\omega$} 
  \threelegs{k}{$\omega$}{$\omega$} 
  \threelegs{l}{$\omega$}{$\omega$} $
  }
\end{center}
if $n>1$. When $n = 1$, there is no impact on cohomology.

\vskip \fsize
\textbf{Lower Excess components.}
In addition to the graphs $\Gamma^{(0)}, \Gamma^{(2)}_{\epsilon}$ and $\Gamma^{(2)}_{\epsilon i}$ already mentioned above, one needs to take into account the generators $\Gamma^{(2)}_{\omega\epsilon\omega\epsilon}, \Gamma^{(2)}_{\omega\epsilon i}$ and $\Gamma^{(2)}_{i;j}$, which contribute a $V_{31^{n-3}}$ term in degree ${d = \frac{3}{2} (g-3) + 14}$, as well as the generator $\Gamma^{(2)}_{ij}$ contributing a $V_{31^{n-3}} \oplus V_{21^{n-2}}$ in degree $ d = \frac{3}{2} (g-1) + 11$.
Once again, the contribution vanishes when $n=1$.

\textbf{Later pages of the spectral sequence.}
Let us now assemble all the results
 above and deduce the cohomology groups.

From the preliminary observation above, we see that for $n>1$, the only possible cancellation on the later pages of the spectral sequence are between $V_{31^{n-3}}$ and $V_{21^{n-2}}$ terms represented by $\Gamma^{(2)}_{i;j}$, $\Gamma^{(2)}_{ij}$, $\Gamma^{(21^2)}_{ \epsilon;i;j}$, and $\Gamma^{(2^2)}_{ \epsilon k ;ji}$. And indeed, we have 
$\Gamma^{(21^2)}_{ \epsilon;i;j} \in \Im \d_{\omega} (\Gamma^{(2)}_{i;j})$
and
$\Gamma^{(2^2)}_{ \epsilon k ;ji} \in \Im \d_{\omega} (\Gamma^{(2)}_{ij})$.

In the case $n=1$, the cancellation on later pages of the spectral sequence comes from 
$\Gamma^{(21^2)}_{ \epsilon i} \simeq \Gamma^{(21^2)}_{ \epsilon} \overset{\d_{\omega}}{\mapsto} \Gamma^{(2^2)}_{ \epsilon;\epsilon} + (\ldots) $
as well as from
$\Gamma^{(31)}_{ \epsilon\epsilon;i} \overset{\d_{\omega}}{\mapsto} \Gamma^{(4)}_{\epsilon \cdot i} + (\ldots) .$

Finally, taking into account the variations in small genus as well as small $n$, we obtain the listed cohomology groups for the excess four graph complexes.
\end{proof}

\begin{remark}
  It would be interesting to obtain similar explicit results for a wider range of $(g,n)$. However, we note that our approach hardly enables computation by hand for any higher excess. Reason for this are the exponential growth of the complexity of computation with excess (see Remark \ref{rmk:expgrowth}), as well as properties of the generating graphs related to excess such as Lemma 4.4. in \cite{PW23}. As a consequence, an algorithmic support or more efficient tools would be required.
\end{remark}

\bibliographystyle{amsalpha}
\bibliography{citation}
	
\end{document}